\documentclass[12pt]{article}
\usepackage{amsmath,amssymb,amsthm, graphicx}
\usepackage[all]{xy}
\newtheorem{thm}{Theorem}
\newtheorem*{thm*}{Theorem}

\newtheorem{lem}[thm]{Lemma}
\theoremstyle{definition}

\theoremstyle{remark}
\newtheorem{rmk}[thm]{Remark}

\def\co{\colon\thinspace}
\newcommand{\mb}[1]{\mathbb{#1}}

\newcommand{\F}{\mathbb F}

\newcommand{\Ext}{\ensuremath{{\rm Ext}}}
\newcommand{\Sq}{{\rm Sq}}
\newcommand{\THH}{{\rm THH}}
\newcommand{\HF}{{\rm H}{\mb F}}
\newcommand{\smsh}[1]{\ensuremath{\mathop{\wedge}_{#1}}}

\newcommand{\xym}[1]{
\vskip 0.7pc
\centerline{\xymatrix{#1}}
\vskip 0.7pc
}

\title{The spectra $ko$ and $ku$ are not Thom spectra: an approach using $\THH$}
\author{Vigleik Angeltveit, Michael Hill, Tyler Lawson}

\begin{document}
\maketitle

\begin{abstract}
We apply an announced result of Blumberg-Cohen-Schlichtkrull to
reprove (under restricted hypotheses) a theorem of Mahowald: the
connective real and complex $K$-theory spectra are not Thom spectra.
\end{abstract}

The construction of various bordism theories as Thom spectra served as
a motivating example for the development of highly structured ring
spectra.  Various other examples of Thom spectra followed; for
instance, various Eilenberg-Maclane spectra are known to be
constructed in this way \cite{goodmahowald}.  However, Mahowald proved
that the connective $K$-theory spectra $ko$ and $ku$ are not the
$2$-local Thom spectra of any vector bundles, and that the spectrum
$ko$ is not the Thom spectrum of a spherical fibration classified by a
map of H-spaces \cite{evilmahowald}.  Rudyak later proved that $ko$
and $ku$ are not Thom spectra $p$-locally at odd primes $p$
\cite{rudyak}.

There has been a recent clarification of the relationship between Thom
spectra and topological Hochschild homology.  Let $BF$ be the
classifying space for stable spherical fibrations.

\begin{thm*}[Blumberg-Cohen-Schlichtkrull \cite{bcs}]
If $Tf$ is a spectrum which is the Thom spectrum of a 3-fold
loop map $f\co X \to BF$, then there is an equivalence
\[
\THH(Tf) \simeq Tf \smsh{} BX_+.
\]
\end{thm*}

(Here $\THH(Tf)$ is the topological Hochschild homology of the Thom
spectrum $Tf$, which inherits an $E_3$-ring spectrum structure
\cite[Chapter IX]{lms}.)  Paul Goerss asked whether this theorem could
be combined with the previous computations of the authors \cite{ahl}
to give a proof that $ku$ and $ko$ are not Thom spectra under this
``$3$-fold loop'' hypothesis.  This paper is an affirmative answer to
that question.

The forthcoming Blumberg-Cohen-Schlichtkrull paper includes a more
careful analysis of the topological Hochschild homology of Thom
spectra in the case of $1$-fold and $2$-fold loop maps, and should
provide weaker conditions for these results to hold.  However, in
order to construct $\THH$ one must assume that the Thom spectrum has
some highly structured multiplication, which is not part of the
assumptions in Mahowald's original proof that $ko$ is not a Thom
spectrum.

\section{The case of $ku$}

Assume that $ku$, $2$-locally, is the Thom spectrum $Tf$ of a $3$-fold
loop map.  We then obtain an equivalence:
\[
\THH(ku) \simeq ku \smsh{} X_+ \simeq ku \smsh{ko} (ko \smsh{} X_+)
\]
Splitting off a factor of $ku$ from the natural unit $S^0 \to X_+$, it
thus suffices to show there is no $ko$-module $Y$ such that smashing
over $ko$ with $ku$ gives the reduced object $\overline \THH(ku)$.

The homotopy of $\overline \THH(ku)$ in degrees below 10 has
$ku_*$-module generators $\lambda_1$ and $\lambda_2$ in degrees 3 and
7 respectively, subject only to the relation $2 \lambda_2 = v_1^2
\lambda_1$ for $v_1$ the Bott element in $\pi_2 ku$ \cite{ahl}.  A
skeleton for such a complex $Y$ could be constructed with cells in
degree 3, 7, and 8.

If we had such a $ko$-module $Y$, we could iteratively construct maps
\[
\Sigma^3 ko \to \Sigma^3 ko \vee \Sigma^7 ko \to (\Sigma^3 ko \vee
\Sigma^7 ko) \cup_\phi (C \Sigma^7 ko) \to Y
\]
by attaching a 3-cell, a 7-cell (which has 0 as the only possible
attaching map), and an 8-cell via some attaching map $\phi$.

However, this requires us to lift the attaching map for the 8-cell
along the map
\xym{
\pi_7 (\Sigma^3 ko \vee \Sigma^7 ko) \ar[d] \ar@{=}[r] & \mb Z \oplus
\mb Z \ar[d]_{(2,1)} \\
\pi_7 (\Sigma^3 ku \vee \Sigma^7 ku) \ar@{=}[r] & \mb Z \oplus
\mb Z.
}
The element we need to lift is $(v_1^2,2)$, but the image is generated
by $(2v_1^2, 0)$ and $(0,1)$.

This contradiction is essentially the same as that given by Mahowald
assuming that $ku$ is the Thom spectrum of a spherical fibration on a
$1$-fold loop space \cite{evilmahowald}.

\begin{rmk}
The analogue of this argument fails for the Adams summand at odd
primes.  The essential difference is that at odd primes, the element
$v_1^p$ in the Adams-Novikov spectral sequence is a nullhomotopy of
$p^2$ times the $p$'th torsion generator in the image of the
$J$-homomorphism, whereas at $p=2$ the element $v_1^2$ is a
nullhomotopy of $4\nu + \eta^3$.
\end{rmk}

\section{The case of $ko$}

Similarly to the previous case, suppose that we had $\THH(ko) \simeq
ko \smsh{} Y_+$ for a space $Y$, and hence the reduced object
satisfies $\overline \THH(ko) \simeq ko \smsh{} Y$.  Then
\[
\overline \THH(ko;\HF_2) \simeq \HF_2 \smsh{ko} \overline \THH(ko)
\simeq \HF_2 \smsh{} Y.
\]
The ${\cal A}_*$-comodule structure on $H_*(Y)$ would then be a lift
of the coaction of ${\cal A}(1)_* = \pi_* (\HF_2 \smsh{ko} \HF_2)$ on
$\overline \THH_*(ko;\HF_2)$.  In particular, this determines the
action of $\Sq^1$ and $\Sq^2$.

The groups $\THH_*(ko;\HF_2)$ through degree 20 have
generators in degree $0$, $5$, $7$, $8$, $12$, $13$, $15$, $16$, and
$20$. The groups as a module over $\mathcal A(1)$ are presented in
Figure~\ref{fig:HTHHko}. In this, dots represent generators of the
corresponding group, straight lines represent the action of $Sq^1$,
curved lines represent $Sq^2$, and the box indicates that the entire
picture repeats polynomially on the class in degree $16$.

\begin{figure}[ht!]
\centering
\includegraphics{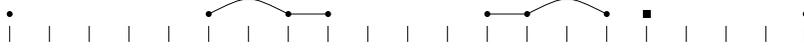}
\caption{$\pi_\ast(THH(ko;H\F_2))$ as an $\mathcal A(1)$-module}
\label{fig:HTHHko}
\end{figure}

\begin{lem}
Suppose that there was a lift of the $20$-skeleton of $\overline
\THH(ko)$ to a spectrum $W$ with cells in degrees $5$, $7$, $8$, $12$,
$13$, $15$, $16$, and $20$.  Then the attaching map for the $16$-cell
over the sphere would be $2\nu$-torsion.
\end{lem}

\begin{proof}
This is a consequence of the calculations of \cite{ahl}, as follows.
Modulo the image of the $13$-skeleton, the reduced
object $\overline \THH(ko)$ has cells in degrees $15$, $16$, and
$20$, with the generator in degree $16$ attached to $4$ times
the generator in degree $15$ and the generator in degree $20$ attached
to $2v_1^2$ times the generator in degree $15$.

However, the Hurewicz map $\mb S/4 \to ko/4$ is an isomorphism on
$\pi_4$, and so any lift of the attaching map for the $20$-cell would
have to lift to a generator of $\pi_{19} (\Sigma^{15} \mb S/4)$.
However, the image of this generator modulo the $15$-skeleton is the
element $2\nu \in \pi_{19}(\Sigma^{16} \mb S)$.  This forces the
attaching map for the $16$-cell to be $2\nu$-torsion, as desired.
\end{proof}

We now apply this to show the nonexistence of such a spectrum by
assuming that we have already constructed a $16$-skeleton for it.

\begin{thm}
Suppose that we have ($2$-locally) a suspension spectrum $Z$ of a space
such that $ko \smsh{} Z$ agrees with $\overline \THH(ko)$ through
degree $19$, with cells in degrees $5$, $7$, $8$, $12$, $13$, $15$, and
$16$.  The attaching map for the next necessary cell (in degree $20$)
does not lift to the homotopy of $Z$. 
\end{thm}

\begin{proof}
Let $S$ be the $15$-skeleton of $Z$, and $U$ the $8$-skeleton.  There
exists a cofiber sequence
\[
U \to S \to Q \to \Sigma U
\]
where $U$ is the unique connective spectrum whose homology is an
``upside-down question mark'' starting in degree 5, and $Q$ is the
unique connective spectrum whose homology is a ``question mark''
starting in degree 12.  (For this reason, the spectrum $S$ is
informally called the ``Spanish question.'')  By the previous lemma,
it suffices to show that any attaching map for the $16$-cell cannot be
$2\nu$-torsion.

The following charts display the final results of the Adams spectral
sequence for the homotopy of $U$ (Figure~\ref{fig:U}) and $Q$
(Figure~\ref{fig:Q}). The nontrivial differentials for $U$ are deduced
from corresponding differentials for the sphere.

\begin{figure}[ht]
\centering
\includegraphics{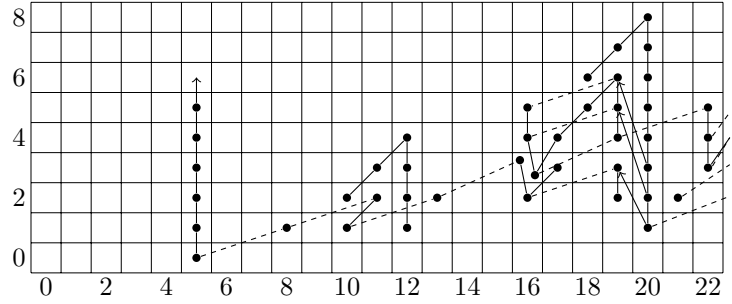}
\caption{The Adams Spectral Sequence for $U$}
\label{fig:U}
\end{figure}

\begin{figure}[ht]
\centering
\includegraphics{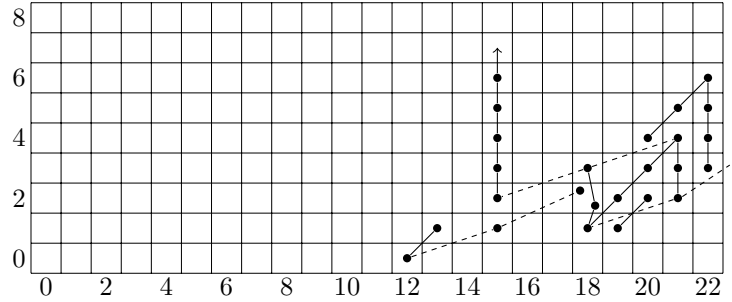}
\caption{The Adams Spectral Sequence for $Q$}
\label{fig:Q}
\end{figure}

We note two things about the homotopy of $U$.

\begin{itemize}
\item First, by comparison with the sphere, there are no hidden
multiplication-by-$4$ extensions in total degree $19$.  The image of
$\pi_{19} \Sigma^5 \mb S$ is an index $2$ subgroup isomorphic to $(\mb
Z/2)^2$.

\item Second, let $x$ be any class in total degree $11$.  As
$\eta$-multiplication surjects onto degree $11$, we would have $x =
\eta y$ for some $y$ in total degree $10$.  However, then as
$\eta$-multiplication is surjective onto total degree $17$ we would
have $\sigma y = \eta z$ for some $z$, and therefore
\[
\sigma \eta x = \eta^3 z = 4\nu z.
\]
However, by the previous note there can be no hidden
multiplication-by-$4$ extensions in degree $19$, so $\sigma \eta x =
0$.
\end{itemize}

The attaching map $f\co \Sigma^{-1} Q \to U$ for $S$ must be a
lift of the corresponding $ko$-module attaching map $ko \smsh{} f\co
\Sigma^{-1} ko \smsh{} Q \to ko \smsh{} U$ for $ko \smsh{} S$.  We
display here the Adams charts computing the homotopy groups of the
function spectra parametrizing the possible attaching maps.

Figure~\ref{fig:DQU} displays the Adams spectral sequence chart for
the homotopy of $F(\Sigma^{-1} Q, U) \simeq \Sigma DQ \smsh{} U$.

\begin{figure}[ht!]
\centering
\includegraphics{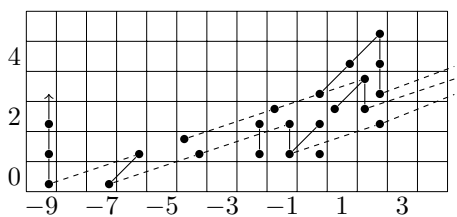}
\caption{The Adams $E_2$-term for $F(\Sigma^{-1}Q,U)$}
\label{fig:DQU}
\end{figure}

The Adams spectral sequence chart for $F_{ko}(\Sigma^{-1} ko\smsh{} Q,
ko \smsh{} U) \simeq ko \smsh{} F(\Sigma^{-1} Q, U)$ is shown in
Figure~\ref{fig:koDQU}.

\begin{figure}[ht!]
\centering
\includegraphics{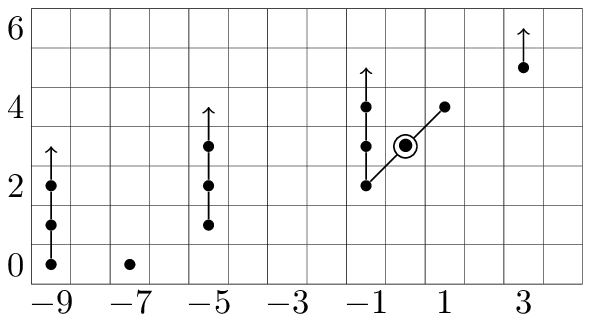}
\caption{The Adams $E_2$-term for $ko\smsh{}F(\Sigma^{-1} Q, U)$}
\label{fig:koDQU}
\end{figure}

We note that there is a unique nontrivial attaching map over $ko$; the
homotopy computations of \cite{ahl} show that the attaching
map $ko \smsh{} f$ must be the unique nontrivial element in $\pi_0$ of
$ko\smsh{}F(\Sigma^{-1}Q,U)$. In the figure, this class is circled.
The lift to the sphere must be of Adams filtration $2$ or higher, as a
lift of Adams filtration $1$ would give the cohomology of $S$ visible
squaring operations $\Sq^8$ out of dimensions below $8$.

We then note that the product $(ko \smsh{} f)\eta$ is nontrivial, and
lifts to the unique map $(f\eta)$ over the sphere which is an
$\eta$-multiple.  It has Adams filtration $4$.

Figure~\ref{fig:Shomotopy} is an Adams spectral sequence chart for the
homotopy of $S$.  The indicated arrows are {\em not} necessarily
differentials; they describe the unique nontrivial map $g\co
\Sigma^{-1} Q \to U$ of Adams filtration 3 in $\Ext$.  We note that
$g$ and $f$ agree on multiples of $\eta$, and so these do describe
$d_3$ differentials on multiples of $\eta$.

\begin{figure}[ht!]
\centering
\includegraphics{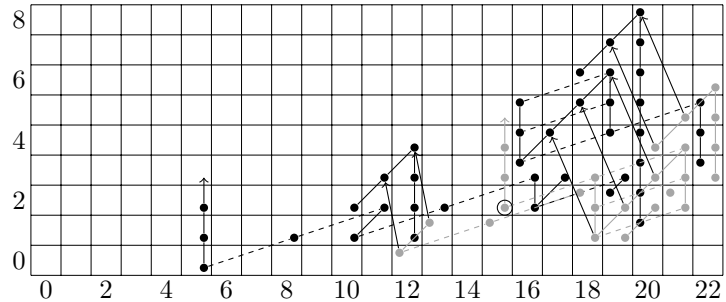}
\caption{The Adams $E_2$-term for $S$, with map of filtration 3}
\label{fig:Shomotopy}
\end{figure}

In particular, there must be a $d_3$ differential out of degree $(t-s,s)
= (19,2)$.  By comparing with the spectral sequences for $Q$ and $U$,
we find that the only other possible differential supported on a class
in total degree $19$ would be a $d_5$ on the class in degree
$(19,1)$.  However, this class is $\sigma y$ for the class $y$ in
bidegree $(12,0)$, and as previously noted we must have $\sigma \eta
f(y) = 0$ where $f$ is the attaching map.  Therefore, the specified
$d_5$ differential does not exist and the class in degree $(19,1)$
survives to homotopy.

Figure~\ref{fig:koShomotopy} describes the Adams $E_3$ page for the homotopy
of $ko \smsh{} S$.  The indicated differentials are the image of $ko
\smsh{} g = ko \smsh{} f$.

\begin{figure}[ht!]
\centering
\includegraphics{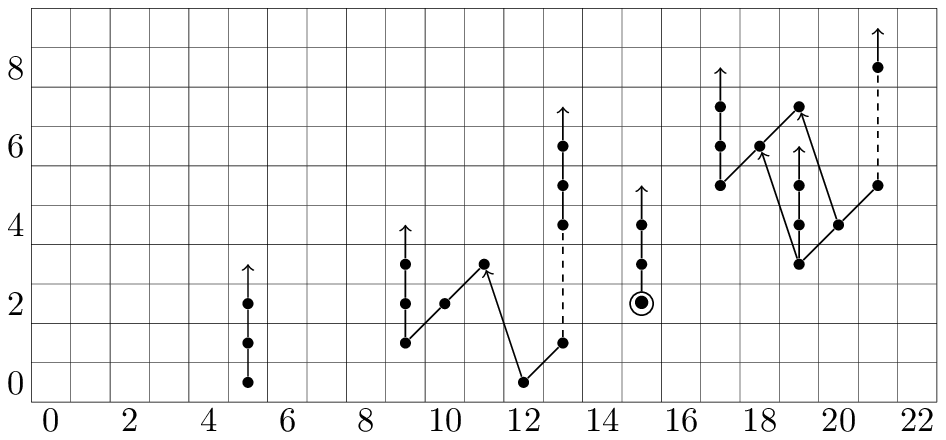}
\caption{The Adams $E_3$-term for $ko \smsh{} S$}
\label{fig:koShomotopy}
\end{figure}

Comparing these, we find that the (marked) attaching map $ko \smsh{}
h$ for the $16$-cell has two possible lifts to a map $h$ over the
sphere up to multiplication by a $2$-adic unit: there is one map in
Adams filtration $1$ and one map in Adams filtration $2$.  These two
lifts differ by a $2$-torsion element (as the image is torsion-free),
and so the element $2\nu h$ is uniquely defined.  One possible choice
of $h$ is marked in Figure~\ref{fig:Shomotopy}.

We claim that there is a hidden extension
\[
2\nu h \neq 0.
\]
As a result, by the previous lemma the attaching map for the $20$-cell
cannot possibly lift.

The image of $\nu h$ is $2$-torsion in $\pi_* Q$, and hidden
multiplication-by-2 can be detected on Ext by the Massey
product $\langle f,\nu h,2 \rangle$.  Multiplying this by $\eta$, we
find
\[
\langle f,\nu h,2\rangle \eta = f(\langle \nu h,2,\eta\rangle).
\]
However, the Massey product $\langle \nu h,2,\eta\rangle$ is the
nontrivial element in bidegree $(20,4)$ in $\pi_* Q$: the
element $\langle \nu h,2,\eta\rangle \eta$ has a nontrivial image under
$f$, and therefore so does $\langle \nu h,2,\eta\rangle$.

(The indeterminacy in the element $f(\langle \nu h,2,\eta\rangle)$
consists of elements $f(y\eta)$ for $y \in \pi_* Q$.  The only nonzero
such image, however, is an element in $\pi_* U$ of bidegree $(19,6)$,
as we ruled out the possibility that the element in bidegree $(20,1)$
has nonzero image under $f$.)
\end{proof}

\nocite{*}
\bibliography{thom}

\end{document}